\documentclass[11pt]{amsart}
\usepackage{enumerate}
\usepackage{amssymb,pb-diagram}
\usepackage{amsmath}
\usepackage{amsthm,amssymb,amsfonts}
\usepackage{amssymb,amscd}
\usepackage{epic,eepic,epsfig}
\usepackage{latexsym}
\usepackage{xy}
\usepackage{color}
\xyoption{all}

\newtheorem{Thm}{Theorem}[section]

\newtheorem{Lem}[Thm]{Lemma}
\newtheorem{Prop}[Thm]{Proposition}
\newtheorem{Claim}[Thm]{Claim}

\newtheorem{proposition}[Thm]{Proposition}
\newtheorem{lemma}[Thm]{Lemma}

\theoremstyle{definition}

\newtheorem{Remark}[Thm]{Remark}

\theoremstyle{remark}

\def\ldots{\mathinner{\ldotp\ldotp\ldotp}}
\def\cdots{\mathinner{\cdotp\cdotp\cdotp}}

\newcommand{\nn}{\mathbb{N}}

\def \cal{\mathcal}

\def \diam{\text{diam }}

\def\eps{\varepsilon}

\def\Mdb{\mathbb M}
\def\Ndb{\mathbb N}

\def\Rdb{\mathbb R}

\def\supp{\text{supp}}

\newcommand{\bib}{\bibitem}

\def\supp{\text{supp }}

\begin{document}

\title{Prescribed Szlenk index of separable Banach spaces}

\author{R.M.~Causey}
\address{R.M.~Causey, Department of Mathematics, Department of Mathematics, Miami University, Oxford, OH 45056, USA}
\email{causeyrm@miamioh.edu}

\author{G.~Lancien}
\address{Laboratoire de Math\'ematiques de Besan\c con, Universit\'e Bourgogne Franche-Comt\'e, CNRS UMR-6623, 16 route de Gray, 25030 Besan\c con C\'edex, Besan\c con, France}
\email{gilles.lancien@univ-fcomte.fr}

%\date{\today}

\thanks{The second named author was supported by the French ``Investissements d'Avenir'' program, project ISITE-BFC, contract
 ANR-15-IDEX-03.}

\maketitle

\begin{abstract} In a previous work, the first named author described the  set $\cal P$ of all values of the Szlenk indices of separable Banach spaces. We complete this result by showing that for any integer $n$ and any ordinal $\alpha$ in $\cal P$, there exists a separable Banach space $X$ such that the Szlenk of the dual of order $k$ of $X$ is equal to the first infinite ordinal $\omega$ for all $k$ in $\{0,..,n-1\}$ and equal to $\alpha$ for $k=n$. One of the ingredients is to show that the Lindenstrauss space and its dual both have a Szlenk index equal to $\omega$. We also show that any element of $\cal P$ can be realized as a Szlenk index of a reflexive Banach space with an unconditional basis.

\end{abstract}

\section{Introduction and notation}

In this paper we exhibit some new properties of the Szlenk index, an ordinal index associated with a Banach space. More precisely we study the values that can be achieved as a Szlenk index of a Banach space and of its iterated duals. Let us first recall the definition of the Szlenk index.\\
Let $X$ be a Banach space, $K$ a weak$^*$-compact subset of its dual $X^*$ and $\eps>0$.
Then we define
$$s_\eps^1(K)=\{ x^* \in K,\ {\rm for\ any}\ {\rm weak^*-neighborhood}\ U\ {\rm of}\ x^*,\  \diam (K \cap U) \ge \varepsilon\}$$
and inductively the sets $s_\eps^\alpha(K)$ for $\alpha$ ordinal as follows: $s_\eps^{\alpha+1}(K)=s_\eps^1(s_\eps^\alpha(K))$ and $s_\eps^\alpha(K)=\bigcap_{\beta<\alpha}s_\eps^\beta(K)$ if $\alpha$ is a limit ordinal.\\
Then $Sz(K,\eps)=\inf\{\alpha,\ s^\alpha_\eps(K)=\emptyset\}$ if it exists and we denote $Sz(K,\eps)=\infty$ otherwise. Next we define $Sz(K)=\sup_{\eps>0}Sz(K,\eps)$.
The closed unit ball of $X^*$ is denoted $B_{X^*}$ and the Szlenk index of $X$ is $Sz(X)=Sz(B_{X^*})$.

The Szlenk index  was first introduced by W. Szlenk \cite{Szlenk1968}, in a slightly different form, in order to prove that there is no separable reflexive Banach space universal for the class of all separable reflexive Banach spaces. The key ingredients in \cite{Szlenk1968} are that the Szlenk index of a separable reflexive space is always countable and that for any countable ordinal $\alpha$, there exists a separable reflexive Banach space with Szlenk index larger than $\alpha$. It has been remarked in \cite{Lancien1996} that, when it is different from $\infty$, the Szlenk index of a Banach space is always of the form $\omega^\alpha$, for some ordinal $\alpha$. Here, $\omega$ denotes the first infinite ordinal. On the other hand, it follows from the work of Bessaga and Pe\l czy\'nski \cite{BessagaPelczynski1960} and Samuel \cite{Samuel1983} that if $K$ is an infinite, countable, compact topological space, then the Szlenk index of the space of continuous functions on $K$ is $\omega^{\alpha+1}$, where $\alpha$ is the unique countable ordinal such that $\omega^\alpha\le CB(K) <\omega^{\alpha+1}$ and  $CB(K)$ is the Cantor-Bendixson index of $K$. Finally, the set of all possible values for the Szlenk index of a Banach space was completely described in \cite{Causey2017a} (Theorem 1.5). One consequence of this general result is that for any countable ordinal $\alpha$, there exists an infinite dimensional separable Banach space $X$ with $Sz(X)=\alpha$ if and only if
$\alpha\in \Gamma\setminus \Lambda$, where
$$\Gamma=\{\omega^\xi,\ \xi  \in [1,\omega_1)\}\ \  \text{and} \ \ \Lambda=\{\omega^{\omega^\xi},\ \xi \in [1,\omega_1)\ \text{ and}\ \xi \text{\ is  a limit ordinal}\}.$$

Our first result shows that there is quite some freedom in prescribing the Szlenk indices of the iterated duals of a separable Banach space. We shall use the notation $Z^{(n)}$ for the $n^{\rm th}$ dual of a Banach space $Z$. Then our statement is the following.

\begin{Thm}\label{main}Let $n\in \Ndb$ and $\alpha \in \Gamma\setminus \Lambda$. Then there exists a separable Banach space $Z_n$ such that for all $k\in \{0,..,n-1\}$, $Sz(Z_n^{(k)})=\omega$ and $Sz(Z_n^{(n)})=\alpha$.
\end{Thm}

The above result relies on a  statement that has its own interest. Let us first recall that in \cite{Lindenstrauss1971}, J. Lindenstrauss constructed, for any separable Banach space $X$, a Banach space $Z$ such that $Z^{**}/Z$ is isomorphic to $X$. We prove the following.

\begin{Thm}\label{SzlenkofZZ*} For any separable Banach space $X$, the associated Lindenstrauss space $Z$ satisfies the following property: $Sz(Z)=Sz(Z^*)=\omega$.
\end{Thm}

Theorem \ref{SzlenkofZZ*} and then Theorem \ref{main} are proved in section 2. In section 3, we show the following refinement of Theorem 1.5 from \cite{Causey2017a}.

\begin{Thm}\label{Causey2} For any $\alpha \in \Gamma\setminus \Lambda$ there exists a separable reflexive Banach space $G_\alpha$ with an unconditional basis such that $Sz(G_\alpha)=\alpha$ and $Sz(G_\alpha^*)=\omega$.
\end{Thm}

We conclude this introduction by recalling the definitions of some uniform asymptotic properties of norms that we will use. For a Banach space $(X,\|\ \|)$ we
denote by $B_X$ the closed unit ball of $X$ and by $S_X$ its unit
sphere. The following definitions are due to V. Milman \cite{Milman} and we follow the notation from \cite{JohnsonLindenstraussPreissSchechtman2002}. For $t\in [0,\infty)$, $x\in S_X$ and $Y$ a closed linear subspace of $X$, we define
$$\overline{\rho}_X(t,x,Y)=\sup_{y\in S_Y}\big(\|x+t y\|-1\big)\ \ \ \ {\rm and}\ \ \
\ \overline{\delta}_X(t,x,Y)=\inf_{y\in S_Y}\big(\|x+t y\|-1\big).$$ Then
$$\overline{\rho}_X(t,x)=\inf_{{\rm
dim}(X/Y)<\infty}\overline{\rho}_X(t,x,Y)\ \ \ \ {\rm and}\ \ \
\ \overline{\delta}_X(t,x)=\sup_{{\rm
dim}(X/Y)<\infty}\overline{\delta}_X(t,x,Y).$$
Finally
$$\overline{\rho}_X(t)=\sup_{x\in S_X}\overline{\rho}_X(t,x)\ \ \ \ {\rm and}\ \ \ \
\overline{\delta}_X(t)=\inf_{x\in S_X}\overline{\delta}_X(t,x).$$
The norm $\|\ \|$ is said to be
{\it asymptotically uniformly smooth} (in short AUS) if
$$\lim_{t \to 0}\frac{\overline{\rho}_X(t)}{t}=0.$$
It is said to be {\it asymptotically uniformly convex} (in short
AUC) if $$\forall t>0 \ \ \ \ \overline{\delta}_X(t)>0.$$
Let $p\in (1,\infty)$ and $q\in [1,\infty)$.\\
We say that the norm of $X$ is {\it $p$-AUS} if there exists $c>0$ such that for all $t\in [0,\infty)$, $\overline{\rho}_X(t)\le ct^p$.\\
We say that the norm of $X$ is {\it $q$-AUC} if there exits $c>0$ such that for all $t \in [0,1]$, $\overline{\delta}_X(t)\ge ct^q$.\\
Similarly, there is on $X^*$ a modulus of weak$^*$
asymptotic uniform convexity defined by
$$ \overline{\delta}_X^*(t)=\inf_{x^*\in S_{X^*}}\sup_{E}\inf_{y^*\in S_E}\big(\|x^*+ty^*\|-1\big),$$
where $E$ runs through all weak$^*$-closed subspaces of
$X^*$ of finite codimension. The norm of $X^*$ is said to be {\it weak$^*$ uniformly asymptotically convex} (in short weak$^*$-AUC) if $\overline{\delta}_X^*(t)>0$ for all $t$ in $(0,\infty)$. If there exists $c>0$ and $q\in [1,\infty)$ such that for all $t\in [0,1]$ $\overline{\delta}_X^*(t)\ge ct^q$, we say that the norm of $X^*$ is $q$-weak$^*$-AUC.

\medskip We will need the following classical duality result concerning these moduli (see for instance \cite{DKLR2017} Corollary 2.3 for a precise statement).

\begin{Prop}\label{duality} Let $X$ be a Banach space.\\
Then $\|\ \|_X$ is AUS if and and only if $\|\ \|_{X^*}$ is weak$^*$-AUC.\\
If $p,q\in (1,\infty)$ are conjugate exponents, then $\|\ \|_X$ is $p$-AUS if and and only if $\|\ \|_{X^*}$ is $q$-weak$^*$-AUC.
\end{Prop}

Finally let us recall the following fundamental result, due to Knaust, Odell and Schlumprecht \cite{KnaustOdellSchlumprecht1999}, which relates the existence of equivalent asymptotically uniformly smooth norms and the Szlenk index.
\begin{Thm}[Knaust-Odell-Schlumprecht]\ \\
 Let $X$ be a separable infinite dimensional Banach space. Then $X$ admits an equivalent norm which is asymptotically uniformly smooth if and only if $Sz(X)=\omega$.
\end{Thm}

\section{Prescribed Szlenk index of iterated duals}

\subsection{Renormings of the Lindenstraus space and of its dual}\

We recall the construction given by J. Lindenstrauss in \cite{Lindenstrauss1971} (see also \cite{LindTzafbook1} Theorem 1.d.3) and introduce notation that will be used throughout this section. We refer the reader to the textbooks \cite{LindTzafbook1} and \cite{AlbiacKaltonbook} for a presentation of the standard notions of a Schauder, shrinking, boundedly complete or unconditional basis of a Banach space.

Let $(X,\|\ \|_X)$ be a separable Banach space. Assume $X\neq \{0\}$ and fix $(x_i)_{i=1}^\infty$, a dense sequence in the unit sphere $S_X$ of $X$. Let $E$ be defined by
$$E=\Big\{a=(a_i)_{i=1}^\infty \in \Rdb^{\Ndb},\ \ \|a\|_E=\sup_{0=p_0<p_1<..<p_k}\Big(\sum_{j=1}^k\big\|\sum_{i=p_{j-1}+1}^{p_j} a_ix_i\big\|^2_X\Big)^{1/2}<\infty\Big\}.$$
Then $(E,\|\ \|_E)$ is a Banach space. Let us denote by $(e_i)_{i=1}^\infty$ the canonical algebraic basis of $c_{00}$, the space of finitely supported real valued sequences. It is clear that $(e_i)_{i=1}^\infty$ is a boundedly complete basis of $E$. It follows that $E$ is isometric to the dual $Y^*$ of a Banach space $Y$ with a shrinking basis. If $(e_i^*)_{i=1}^\infty$ is the sequence of coordinate functionals associated with the basis $(e_i)_{i=1}^\infty$ of $E$, then the canonical image of $Y$ in its bidual $Y^{**}$ is the closed linear span of $\{e_i^*,\ i\ge 1\}$ and $(e_i^*)_{i=1}^\infty$ can be seen as a shrinking basis of $Y$.\\
Note now that if $a=(a_i)_{i=1}^\infty \in E$, then the series $\sum_{i=1}^\infty a_ix_i$ is converging in $X$. It is important to note that the density of $(x_i)_{i=1}^\infty$ in $S_X$ implies that the map $Q:E\to X$, defined by $Q(a)=\sum_{i=1}^\infty a_ix_i$ is linear, onto, satisfies $\|Q\|=1$ and also that the open mapping constant of $Q$ is one. Consequently, we have that $Q^*$ is an isometry from $X^*$ into $Y^{**}$. The main result of \cite{Lindenstrauss1971} is that $Y^{**}=\widehat{Y}\oplus Q^*(X^*)$, where $\widehat{Y}$ is the canonical image of $Y$ in $Y^{**}$, and the projection from $Y^{**}$ onto $Q^{*}(X^*)$ with kernel $\widehat{Y}$ has norm one. In particular, $Y$ is isomorphic to the quotient space $Y^{**}/Q^*(X^*)$.\\
Now let $Z$ denote the kernel of $Q$. The space $Z$ is a subspace of $E=Y^*$ and its orthogonal $Z^\perp$ is clearly equal to $Q^*(X^*)$. It follows from the classical duality theory that $Z^*$ is isometric to $Y^{**}/Q^*(X^*)$ and therefore isomorphic to $Y$. If $I$ is the inclusion map from $Z$ into $Y^*$ and $J_Y$ is the canonical injection from $Y$ into $Y^{**}$, an isomorphism from $Y$ onto $Z^*$ is given by $T=I^*J_Y$. Finally, if $J_Z$ is the canonical injection from $Z$ into $Z^{**}$, it is easy to check that $T^*J_Z=Id_Z$. It follows immediately that $Z^{**}/J_Z(Z)$ (or simply $Z^{**}/Z$) is isomorphic to $Y^*/Z$ and therefore to $X$.

\medskip The purpose of this subsection is to prove Theorem \ref{SzlenkofZZ*}. In fact, our result is stronger.

\begin{Thm}\label{2AUSZZ*} For any separable Banach space $X$, the associated Lindenstrauss space $Z$ satisfies the following properties.\\
(i) The space $Z^*$ admits an equivalent norm which is 2-AUS.\\
(ii) The space $Z$ admits an equivalent norm which is 2-AUS.
\end{Thm}

We start with the proof of the easy part (i) which can be precisely stated as follows.

\begin{Prop}\label{SzlenkofZ*} The norm $\|\ \|_E$ is 2-weak$^*$-AUC on $Y^*=E$ and therefore $\|\ \|_Y$ is 2-AUS. In particular, $Z^*$ admits an equivalent norm which is 2-AUS, there exists $C>0$ such that for all $\eps >0$, $Sz(Z^*,\eps)\le C\eps^{-2}$, and $Sz(Y)=Sz(Z^*)=\omega$.
\end{Prop}

This result is an immediate consequence of the following elementary lemma.

\begin{Lem} Let $a,b \in E$ and assume that there exits $k\in \Ndb$ such that the sequence $a$ is supported in $[1,k]$ while the sequence $b$ is supported in $[k+3,\infty)$. Then
$$\|a+b\|_E^2\ge \|a\|_E^2+\|b\|_E^2.$$
\end{Lem}

\begin{proof} Since $a$ is supported in $[1,k]$ we can find a sequence $0=p_0<p_1<..<p_m=k+1$ such that
$$\|a\|_E^2=\sum_{j=1}^m\big\|\sum_{i=p_{j-1}+1}^{p_j} a_ix_i\big\|^2_X.$$
Fix $\eta >0$. Since  $b$ is supported in $[k+3,\infty)$ we can find a sequence $k+1=q_0<q_1<..<q_r$ such that
$$\|b\|_E^2 \ge \sum_{j=1}^r\big\|\sum_{i=q_{j-1}+1}^{q_j} b_ix_i\big\|^2_X - \eta.$$
Let $n_j=p_j$ for $j\in \{0,\ldots,m\}$ and $n_j=q_{j-m}$ for $m\le j\le m+r$. Then
$$\|a+b\|_E^2\ge \sum_{j=1}^{m+r}\big\|\sum_{i=n_{j-1}+1}^{n_j} (a+b)_ix_i\big\|^2_X \ge \|a\|_E^2+\|b\|_E^2-\eta.$$
This finishes the proof.
\end{proof}

We now turn to the proof of part (ii) in Theorem \ref{2AUSZZ*}, which will rely on the following technical lemma.

\begin{Lem}\label{blocks} Assume that $a^1,..,a^N$ are skipped blocks with respect to the basis $(e_i)_{i=1}^\infty$ of $E$, meaning that there exist $0=r_0<r_1<..<r_N$ so that
$$\forall k\in \{1,..,N\},\ \supp(a^k)\subset (r_{k-1},r_k)$$
and denote $\eps_k=\|\sum_{i=1}^\infty a^k_ix_i\|_X$. Then
$$\Big\|\sum_{k=1}^N a^k\Big\|_E \le \sum_{k=1}^N \eps_k+ 2\Big(\sum_{k=1}^N \big\|a^k\big\|_E^2\Big)^{1/2}.$$
\end{Lem}

\begin{proof} Fix $0=p_0<..<p_m$ and assume without lost of generality that $p_m\ge r_N$. Then for $j\in \{1,..,m\}$ we denote
$$A_j=\{k\le N,\ (r_{k-1},r_k)\subset (p_{j-1},p_j]\},\ \ A=\bigcup_{j=1}^m A_j\ \ {\rm and}\ \ B=\{1,\cdots,m\}\setminus A.$$
We first estimate

\begin{align*}
&\Bigl(\sum_{j=1}^m\Big\|\sum_{i=p_{j-1}+1}^{p_j}\Big(\sum_{k\in A}a^k_i\Big)x_i\Big\|_X^2\Bigr)^{1/2}
\leq \sum_{j=1}^m \Bigl\|\sum_{i=p_{j-1}+1}^{p_j} \Bigl(\sum_{k\in A} a^k_i\Bigr)x_i\Bigr\|_X \\
&=\sum_{j=1}^m\Big\|\sum_{k\in A_j}\sum_{i=p_{j-1}+1}^{p_j}a^k_ix_i\Big\|_X \le \sum_{j=1}^m\sum_{k\in A_j}\big\|\sum_{i=p_{j-1}+1}^{p_j}a^k_ix_i\big\|_X
\end{align*}\label{eq1}
and we obtain

\begin{align}\label{eq1}
\Bigl(\sum_{j=1}^m\Big\|\sum_{i=p_{j-1}+1}^{p_j}\Big(\sum_{k\in A}a^k_i\Big)x_i\Big\|_X^2\Bigr)^{1/2}
\le \sum_{j=1}^m\sum_{k\in A_j}\eps_k
\le \sum_{k=1}^N \eps_k.
\end{align}
So we may assume that $B$ is not empty and enumerate $B=\{a^{k(1)},\ldots,a^{k(L)}\}$, with $k(1)<\cdots <k(L)$. Note that for $1\le l \le L$, $\supp(a_{k(l)})\subset (r_{k(l)-1},r_{k(l)})\subset (r_{k(l-1)},r_{k(l)})$ and $(r_{k(l-1)},r_{k(l)})$ is not included in any of the sets $(p_{j-1},p_j]$, for $1\le j\le m$. Then we define $i_0=0$ and for $1\le l\le L$, $i_l=\min\{i,\ p_i\ge r_{k(l)}\}$. From the definition of $B$, we get that $2<i_1<\cdots<i_L$ and for all $l\in \{1,..,L\}$, $p_{i_l-1}<r_{k(l)}\le p_{i_l}$. We can now write
$$\sum_{j=1}^m\Big\|\sum_{i=p_{j-1}+1}^{p_j}\Big(\sum_{k\in B}a^k_i\Big)x_i\Big\|_X^2=\sum_{q=1}^L\sum_{j=i_{q-1}+1}^{i_q}
\Big\|\sum_{i=p_{j-1}+1}^{p_j}\Big(\sum_{l=1}^L a^{k(l)}_i\Big)x_i\Big\|_X^2.$$
Using the convention $a^{k(0)}=0=a^{k(L+1)}$ and the properties of our various sequences we get
\begin{align*}
&\sum_{j=1}^m\Big\|\sum_{i=p_{j-1}+1}^{p_j}\Big(\sum_{k\in B}a^k_i\Big)x_i\Big\|_X^2
=\sum_{q=1}^L\sum_{j=i_{q-1}+1}^{i_q}
\Big\|\sum_{i=p_{j-1}+1}^{p_j}\big(a^{k(q)}_i+a^{k(q+1)}_i)x_i\Big\|_X^2\\
&\le \sum_{q=1}^L \big\|a^{k(q)}+a^{k(q+1)}\big\|^2_E\le 4\sum_{q=1}^L\big\|a^{k(q)}\big\|_E^2\le 4\sum_{k=1}^N \big\|a^k\big\|_E^2,
\end{align*}
which yields
\begin{equation}\label{eq2}
\Big(\sum_{j=1}^m\Big\|\sum_{i=p_{j-1}+1}^{p_j}\Big(\sum_{k\in B}a^k_i\Big)x_i\Big\|_X^2\Big)^{1/2} \le 2\Big(\sum_{k=1}^N \big\|a^k\big\|_E^2\Big)^{1/2}.
\end{equation}
The conclusion of the proof of this lemma now clearly follows from equations (\ref{eq1}) and (\ref{eq2}), a triangle inequality and by taking the supremum over all finite sequences $(p_j)_{j}$.
\end{proof}

Before we proceed with the proof of Theorem \ref{2AUSZZ*}, we need to introduce some notation. For an infinite subset $\Mdb$ of $\Ndb$, we denote $[\Mdb]^{<\omega}$ the set of void or finite  increasing sequences in $\Mdb$. The void sequence is denoted $\emptyset$. For $E\in [\Ndb]^{<\omega}$, we denote $|E|$, the \emph{length} of $E$, defined by $|E|=0$ if $E=\emptyset$ and $|E|=k$ if $E=(n_1,\ldots,n_k)$. For $F=(n_1,\ldots,n_l)$ in $[\Ndb]^{<\omega}$, we write $E \prec F$, if $E=\emptyset$ or $E=(n_1,\ldots,n_k)$, for some $k<l$, and we then say that $E$ is a proper initial segment of $F$. We write $E \preceq F$ if $E <F$ or $E=F$ and we then say that $E$ is an initial segment of $F$. For $E=(n_1,..,n_k)\in [\Ndb]^{<\omega}$ and $n\in \nn$ such that $n>n_k$, $(E,n)$ denotes the sequence $(n_1,..,n_k,n)$, while $(\emptyset,n)$ is $(n)$. For a Banach space $X$, we will call a family $(x_E)_{E\in [\Ndb]^{<\omega}}$ in $X$, a \emph{tree in $X$}. Then a family $(x_E)_{E\in [\nn]^{<\omega}}$ in a Banach space $X$ is said to be a \emph{weakly null tree} if for any $E$ in $[\nn]^{<\omega}$ the sequence $(x_{(E,n)})_{n}^\infty$ is weakly null. If $(x_E)_{E\in [\Ndb]^{<\omega}}$ is a tree in the Banach space $X$ and $\Mdb$ is an infinite subset of $\Ndb$, we call $(x_E)_{E\in [\Mdb]^{<\omega}}$ a \emph{refinement} or a \emph{full subtree} of $(x_E)_{E\in [\Ndb]^{<\omega}}$.

\begin{proof}[Proof of (ii) in Theorem \ref{2AUSZZ*}] Fix $(\eps_n)_{n=0}^\infty$ a sequence in $(0,\infty)$ such that $\sum_{n=0}^\infty \eps_n^2 \le \frac14$. Let $(z_F)_{F\in [\nn]^{<\omega}}$ be a weakly null tree in the unit ball $B_Z$ of $Z$. By extracting a full subtree, we may assume that there exist $0=r_0<r_1<\cdots<r_n<\cdots$ and for any $F\in [\Ndb]^{<\omega}\setminus \{\emptyset\}$ there exist $a^F\in B_E$ so that
$$\forall F=(n_1,\ldots,n_k)\in [\nn]^{<\omega}\setminus \{\emptyset\},\ \supp(a^F)\subset (r_{n_k-1},r_{n_k})\ \ \text{and}\ \  \|a^{F}-z_{F}\big\|_E\le \eps_k.$$
Since $(z_F)_{F\in [\nn]^{<\omega}}$ is included in the kernel of $Q$, the last condition implies that
$$\forall F\in [\Ndb]^{<\omega}\setminus \{\emptyset\},\ \ \big\|\sum_{i=1}^\infty a^{F}_ix_i\big\|_X\le \eps_k.$$
We can therefore apply Lemma \ref{blocks} and the triangle inequality to get that for all $(\lambda_F)_{F\in [\nn]^{<\omega}\setminus \{\emptyset\}}$ in $\Rdb$ and all $F\in [\nn]^{<\omega}\setminus \{\emptyset\}$,
$$\big\|\sum_{\emptyset<G\le F}\lambda_Gz_G\big\|_E\le 2\sum_{\emptyset<G\le F}|\lambda_G| \eps_{|G|} + 2 \Big(\sum_{\emptyset<G\le F} \lambda_G^2\Big)^{1/2}.$$
It then follows from our initial choice of the sequence $(\eps_n)_{n=0}^\infty$ and the Cauchy-Schwarz inequality that
$$\forall F\in [\nn]^{<\omega}\setminus \{\emptyset\},\ \ \ \big\|\sum_{\emptyset<G\le F}\lambda_Gz_G\big\|_E\le  3 \Big(\sum_{\emptyset<G\le F} \lambda_G^2\Big)^{1/2}.$$
In the terminology introduced in \cite{Causey2017c} it means that $Z$ satisfies $\ell_2$ upper tree estimates. It then follows from Theorem 1.1 in \cite{Causey2017c} that $Z$ admits an equivalent norm which is 2-AUS.
\end{proof}

\begin{Remark} Statement (i) in Theorem \ref{2AUSZZ*} can be rephrased as follows. The space $Z^*$ admits an equivalent norm whose dual norm is 2-weak$^*$-AUC. It is important to note that this norm cannot be a dual norm of an equivalent norm on $Z$. Indeed a bidual norm cannot be weak$^*$-AUC unless the space is reflexive (see proposition below). In particular, in Lindenstrauss' construction, the space $Y$ is isomorphic but never isometric to $Z^*$.
\end{Remark}

For the convenience of the reader, we state and prove the elementary fact from which the previous remark follows.

\begin{Prop}\label{nonreflexive}
Let $Z$ be a non reflexive Banach space. Then the norm of $Z^{**}$ is not weak$^*$-AUC.
\end{Prop}

\begin{proof} Assume that the Banach space $Z$ is not reflexive. So, there exists $z^{**} \in S_{Z^{**}} \setminus Z$. Pick $\eps>0$ such that $\eps<d(z^{**},Z)$. Fix $\delta>0$ so that $\eps+\delta<d(z^{**},Z)$ and $E$ a weak$^*$-closed finite codimensional subspace of $Z^{**}$. We can write $E=\cap_{i=1}^n \text{Ker}\,z_i^*$, with $z_i^*\in Z^*$. Fix now $\eta >0$. Then, Goldstine's theorem insures that there exists $z\in B_Z$ such that $|(z^{**}-z)(z_i^*)|<\eta$ for all $i\le n$. If we denote $F$ the linear span of $z_1^*,\ldots,z_n^*$, it follows from elementary duality theory that
$$d(z^{**}-z,E)=\|z^{**}-z\|_{Z^{**}/F^\perp}=\|z^{**}-z\|_{F^*}.$$
So, if $\eta$ was chosen small enough, we get that $d(z^{**}-z,E)<\delta$. Thus we can pick $e^{**}\in E$ such that $\|z-z^{**}-e^{**}\|<\delta$. Note that it implies that $\|e^{**}\|>\eps$.\\
Now, writing $z=z^{**}+e^{**}+z-z^{**}-e^{**}$ and using the fact $z\in B_Z$, we deduce that $\|z^{**}+e^{**}\|\le 1+\delta$. Finally, by convexity, it follows that there exists $\lambda \in (0,1)$ so that $\|\lambda e^{**}\|=\eps$ and $\|z^{**}+\lambda e^{**}\|\le 1+\delta$.
Since $\delta$ could be chosen arbitrarily small, we deduce that for any weak$^*$-closed finite codimensional subspace $E$ of $Z^{**}$:
$$\inf_{y^{**}\in S_{E^{**}}} \|z^{**}+\eps y^{**}\|\le 1,$$
which implies that $\overline{\delta}_{Z^*}^*(\eps)=0$ and finishes our proof.
\end{proof}

\subsection{Proof of Theorem \ref{main}}\

\medskip
We now conclude this section with the proof of Theorem \ref{main}.

\begin{proof} We fix $\alpha \in \Gamma\setminus \Lambda$ and do an induction on $n\in \Ndb$.\\
For $n=2$, let $X_\alpha$ (given by Theorem 1.5 in \cite{Causey2017a}) be a separable Banach space such that $Sz(X_\alpha)=\alpha$. Then denote $Z_2$ the Lindenstrauss space such that $Z_2^{**}/Z_2$ is isomorphic to $X_\alpha$. We have, by Theorem \ref{SzlenkofZZ*} that $Sz(Z_2)=Sz(Z_2^*)=\omega$. Next, using Proposition 2.1 in \cite{BrookerLancien2013} we get that there exists $C>0$ such that
$$\forall \eps>0\ \ \ Sz(Z_2^{**},\eps)\le Sz(Z_2^{**}/Z_2,\frac{\eps}{C})Sz(Z_2,\frac{\eps}{C}) <\alpha.$$
The last inequality follows from the fact that $Sz(Z_2^{**}/Z_2,\frac{\eps}{C})<\alpha$, $Sz(Z_2,\eps)<\omega$ and elementary properties of the multiplication of ordinal numbers. We deduce that $Sz(Z_2^{**})$ is at most $\alpha$ and therefore $Sz(Z_2^{**})=\alpha$, since $Sz(Z_2^{**})\ge Sz(Z_2^{**}/Z_2)=Sz(X_\alpha)=\alpha$.\\
Then we can choose $Z_1=Z_2^*$.\\
Assume now that $n\ge 3$ and that spaces $Z_1,\ldots,Z_{n-1}$ have been constructed with the requisite indices of the duals. Then denote $Z_n$ the Lindenstrauss space such that $Z_n^{**}/Z_n$ is isomorphic to $Z_{n-2}$. We already know that $Sz(Z_n)=Sz(Z_n^*)=\omega$. Since $Sz(Z_{n-2})=\omega$, we can use the fact that having a Szlenk index equal to $\omega$ is a three space property (see \cite{BrookerLancien2013}) to deduce that $Sz(Z_n^{**})=\omega$. Then using elementary facts about duality, we have that for all $k\ge 3$ the space $Z_n^{(k)}$ is isomorphic to $Z_n^{(k-2)}\oplus Z_{n-2}^{(k-2)}$ which implies that  $Sz(Z^{(k)})=\max\{Sz(Z_n^{(k-2)}),Sz(Z_{n-2}^{(k-2)})\}$ (see \cite{Causey2017b}). It now clearly follows that $Sz(Z_n^{(k)})=\omega$ for all $k\in \{0,..,n-1\}$ and $Sz(Z_n^{(n)})=\alpha$.
\end{proof}

\section{Prescribing Szlenk indices of reflexive Banach spaces}

We now turn to the proof of Theorem \ref{Causey2}, which will take a few steps.

\medskip
First we describe a general construction of a Banach space associated with a given Banach space with a Schauder basis, which will be essential in the sequel. As it will be clear, this resembles Lindenstrauss' construction. The crucial difference is that the dense sequence $(x_i)_{i=1}^\infty$ in $X$ will be replaced by a normalized Schauder basis of $X$.

So assume that $(x_i)_{i=1}^\infty$ is a normalized Schauder basis for the Banach space $X$ and denote again $(e_i)_{i=1}^\infty$ the canonical algebraic basis of $c_{00}$. We define $X^{\ell_2}$ as the completion of $c_{00}$ with respect to the norm
$$\big\|\sum_{i=1}^\infty a_i e_i \big\|_{X^{\ell_2}}= \sup\Bigl\{\Bigl(\sum_{i=1}^\infty \big\|\sum_{j=k_{i-1}+1}^{k_i} a_jx_j\big\|_X^2\Bigr)^{1/2}:0\leq k_0<k_1<\ldots\Bigr\}.$$
This construction is presented in section 3 of \cite{OdellSchlumprechtZsak2007} in a more general setting. With the notation from \cite{OdellSchlumprechtZsak2007}, the space $X^{\ell_2}$ is $Z^V(E)$, with $Z=X$, $V=\ell_2$ and $E$ being the finite dimensional decomposition of $X$ into the one dimensional spaces spanned by the basis vectors $(x_i)_{i=1}^\infty$ of $X$. Clearly, the definition of $X^{\ell_2}$ depends on our choice of the basis $(x_i)_{i=1}^\infty$. However, we shall omit reference to this dependence in our notation.

Note first that $(e_i)_{i=1}^\infty$ is a basis for $X^{\ell_2}$ which is an unconditional basis for $X^{\ell_2}$ if $(x_i)_{i=1}^\infty$ is unconditional in $X$. Furthermore, the map $e_i \mapsto x_i$ extends to a well defined linear operator $I:X^{\ell_2}\to X$ of norm one.  Note also that $(e_i)_{i=1}^\infty$ is a bimonotone basis for $X^{\ell_2}$, even if $(x_i)_{i=1}^\infty$ is not bimonotone in $X$.

\begin{proposition}\label{OSZreflexive} Assume that $(x_i)_{i=1}^\infty$ is a shrinking basis of $X$. Then

(i) The space $X^{\ell_2}$ is reflexive. In particular, $(e_i)_{i=1}^\infty$ is a shrinking and boundedly complete basis of $X^{\ell_2}$.

(ii) The space $(X^{\ell_2})^*$ is $2$-AUS. In particular $Sz((X^{\ell_2})^*)=\omega$.
\end{proposition}

\begin{proof} The statement (i) is a particular case of Corollary 3.4 in \cite{OdellSchlumprechtZsak2007}.\\
(ii) Since $(e_i)_{i=1}^\infty$ is shrinking, $(X^{\ell_2})^*$ can be seen as the closed linear span of $\{e_i^*: i\in \nn\}$. Now it is clear that if $x^*, y^*\in (X^{\ell_2})^*$ with $\max \supp(x^*)< \min \supp(y^*)$, then $\|x^*+y^*\|^2 \leq \|x^*\|^2+\|y^*\|^2$. Here, the support is meant with respect to the basis $(e_i^*)_{i=1}^\infty$ of $(X^{\ell_2})^*$. Hence $(X^{\ell_2})^*$ is $2$-AUS and has Szlenk index $\omega$.\\
Note that this also implies that the bidual norm on $(X^{\ell_2})^{**}$ is weak$^*$-AUC and, by Proposition \ref{nonreflexive}, reproves the fact that $X^{\ell_2}$ is reflexive, knowing that $(e_i)_{i=1}^\infty$ is shrinking.
\end{proof}

 Our next proposition provides a crucial estimate for $Sz(X^{\ell_2})$.

\begin{proposition}\label{tech9} Assume that $(x_i)_{i=1}^\infty$ is a shrinking basis of $X$.\\
Then  $Sz(X^{\ell_2})\leq Sz(X)$.

\end{proposition}

Our strategy will be to show that $Sz(X^{\ell_2})\le Sz(\ell_2(X))$, where $\ell_2(X)$ is the space of sequences $(x_n)_{n=1}^\infty$ in $X$ such that $\sum_{n=1}^\infty \|x_n\|_X^2$ is finite, equipped with its natural norm :
$$\big\|(x_n)_{n=1}^\infty\big\|_{\ell_2(X)}= \Big(\sum_{n=1}^\infty \|x_n\|_X^2\Big)^{1/2}.$$
Then the conclusion will follow from the well known fact that $Sz(\ell_2(X))=Sz(X)$ when $X$ is infinite dimensional (see \cite{Brooker2011} for a general study of the behavior of the Szlenk index under direct sums).

Let $M_1$ be the set of all sequences $(y_i^*)_{i=1}^\infty$ in $B_{\ell_2(X^*)}$ such that there exist $n\in \nn$ and $0=k_0<\cdots<k_{n-1}$ with the following properties: for every $1\le i<n$, $y_i^*$ belongs to the linear span of $\{x_j^*,\ k_{i-1}<j\le k_i\}$, $y_n^*$ belongs to the closed linear span of $\{x_j^*,\ j>k_{n-1}\}$ and $y_i^*=0$ for all $i>n$.
Then we denote by $M_2$ the set of all sequences $(y_i^*)_{i=1}^\infty$ in $B_{\ell_2(X^*)}$ such that there exits an infinite sequence $0=k_0<\cdots<k_i<\cdots$ such that for all $i\in \nn$, $y_i^*$ belongs to the linear span of $\{x_j^*,\ k_{i-1}<j\le k_i\}$. Finally, we set $M=M_1\cup M_2$.\\
It is easy to check that $M$ is weak$^*$-compact in $\ell_2(X^*)=\ell_2(X)^*$.\\
Recall that $I:X^{\ell_2}\to X$ denotes the continuous linear map such that $I(e_i)=x_i$ and that $\|I\|=1$, and define $j:M\to (X^{\ell_2})^*$ by
$$\forall y^*=(y_i^*)_{i=1}^\infty \in M,\ \ j(y^*)=\sum_{i=1}^\infty I^*y_i^*.$$
An elementary application of the Cauchy-Schwarz inequality shows that $j$ is well defined and that
$$\forall y^* \in M,\ \ \|j(y^*)\|_{(X^{\ell_2})^*} \le \|y^*\|_{\ell_2(X^*)}.$$
It is also easy to verify that $j$ is weak$^*$ to weak$^*$ continuous.\\
Note that the set $j(M)$ can be less formally described as the set of all $\sum_{j=1}^\infty b_je_j^*$ such that there exists an increasing finite or infinite sequence of blocks of $\Ndb$ $(F_k)_{k\in A}$ so that
$$\sum_{k\in A}\big\|\sum_{j\in F_k} b_jx_j^*\big\|_{X^*}^2 \le 1.$$
So we now consider the weak$^*$-compact subset $K=j(M)$ of $B_{(X^{\ell_2})^*}$. First we will need to show that $K$ is norming for $X^{\ell_2}$. More precisely, we have:

\begin{Claim} There exists a constant $c>0$ such that
$$\forall x\in X^{\ell_2},\ \ \|x\|_{X^{\ell_2}}\ge c\sup_{x^*\in K}x^*(x).$$

\label{hb}
\end{Claim}
\begin{proof} Let $C\ge 1$ be the bimonotonicity constant of the Schauder basis $(x_i)_{i=1}^\infty$ of $X$, let $x=\sum_{i=1}^\infty a_ie_i \in X^{\ell_2}$ and $\eps>0$. Pick $0\le k_0<\cdots<k_n$ such that
$$\Bigl(\sum_{i=1}^n \big\|\sum_{j=k_{i-1}+1}^{k_i} a_jx_j\big\|_X^2\Bigr)^{1/2}\ge \|x\|_{X^{\ell_2}}-\eps.$$
It follows from the Hahn-Banach theorem that for all $1\le i\le n$, there exists $u_i^*\in X^*$ with $\text{supp}(u_i^*)\subset (k_{i-1},k_i]$ and such that
$$u_i^*\Big(\sum_{j=k_{i-1}+1}^{k_i} a_jx_j\Big)=\big\|\sum_{j=k_{i-1}+1}^{k_i} a_jx_j\big\|_X\ \ \text{and}\ \ \|u_i^*\|_{X^*}\le C.$$
We now set
$$y_i^*=\frac{\big\|\sum_{j=k_{i-1}+1}^{k_i} a_jx_j\big\|_X\, u_i^*}{C\Big(\sum_{i=1}^n\big\|\sum_{j=k_{i-1}+1}^{k_i} a_jx_j\big\|_X^2\Big)^{1/2}}\ \ \text{for}\ 1\le i\le n\ \  \text{and}\ y_i^*=0\ \text{for}\ i>n.$$
It is then clear that $y^*=(y_i^*)_{i=1}^\infty \in M$ and
$$j(y^*)(x)= \frac{1}{C}\Big(\big\|\sum_{j=k_{i-1}+1}^{k_i} a_jx_j\big\|_X^2\Big)^{1/2}\ge \frac{\|x\|_{X^{\ell_2}}-\eps}{C}.$$
This finishes the proof of our claim.
\end{proof}

\begin{Claim} The function $j:M\to K$ is $2C$-Lipschitz, where $C$ is the bimonotonicity constant of the basis $(x_i)_{i=1}^\infty$ in $X$.

\label{lip}
\end{Claim}

\begin{proof} Let us fix $y^*=(y^*_i)_{i=1}^\infty, z^*=(z^*_i)_{i=1}^\infty\in M$. Then there exist $S,T\subset \nn$ and sequences of successive intervals $(I_s)_{s\in S}$, $(J_t)_{t\in T}$, where $S,T$ are (possibly infinite) initial segments of $\nn$, $\{i:y^*_i\neq 0\}\subset S$,  $\{i: z^*_i\neq 0\}\subset T$, and for each $s\in S$ and $t\in T$, $\text{supp}(y^*_s)\subset I_s$ and $\text{supp}(z^*_t)\subset J_t$ (here the supports of $y^*_s$ and $z^*_t$ are meant with respect to the basis $(x^*_j)_{j=1}^\infty$ of $X^*$). By allowing either $I_s=\varnothing$ or $J_t=\varnothing$ for $s>\max S$ or $t>\max T$, we may assume $S=T=\nn$.   For each $i\in\nn$, consider three cases: \begin{enumerate}[(a)]\item $J_i\subset I_i$,
\item $I_i\subset J_i$,
\item neither $(a)$ nor $(b)$ holds.
\end{enumerate}
If $(a)$ holds, let $u_i^*=y^*_i-z^*_i\in \text{span}\{x^*_j: j\in I_i\}$ and $v^*_i=0\in \text{span}\{x^*_j: j\in J_i\}$.\\
If $(b)$ holds, let $u^*_i=0\in \text{span}\{x^*_j: j\in I_i\}$ and $v^*_i=y^*_i-z^*_i\in \text{span}\{x^*_j: j\in J_i\}$.\\
If $(c)$ holds, let $u^*_i=P^*_{I_i\setminus J_i}(y^*_i-z^*_i)\in \text{span}\{x^*_j: j\in I_i\}$ and $v^*_i=P^*_{J_i}(y^*_i-z^*_i)\in \text{span}\{x^*_j: j\in J_i\}$. Here, for an interval $I$, $P_I:X\to \text{span}\{x_j: j\in I\}$ denotes the basis projection. Let us note that in case $(c)$, $I_i\setminus J_i$ is an interval. Then, since each vector $u^*_i$, $v^*_i$ is either zero or an interval projection of $y^*_i-z^*_i$, we have that for each $i$,  $\|u_i^*\|_{X^*} \leq C\|y^*_i-z^*_i\|_{X^*}$ and $\|v^*_i\|_{X^*}\leq C\|y^*_i-z^*_i\|_{X^*}$. From this it follows that $u^*=(u^*_i)_{i=1}^\infty, v^*=(v^*_i)_{i=1}^\infty$ lie in $\ell_2(X)^*$ and $\|u^*\|_{\ell_2(X)^*}, \|v^*\|_{\ell_2(X)^*} \leq C\|y^*-z^*\|_{\ell_2(X)^*}$.   Using that the $(u^*_i)_{i=1}^\infty$ are successively supported, another application of the Cauchy-Schwarz inequality yields that $\sum_{i=1}^\infty u^*_i $ is norm convergent in $(X^{\ell_2})^*$ with $\|\sum_{i=1}^\infty u^*_i\|_{(X^{\ell_2})^*} \leq C \|y^*-z^*\|_{\ell_2(X)^*}$. Similarly, $\|\sum_{i=1}^\infty v^*_i\|_{\ell_2(X)^*} \leq C\|y^*-z^*\|_{\ell_2(X)^*}$.   Since $j(y^*)-j(z^*)=\sum_{i=1}^\infty y^*_i-z^*_i=\sum_{i=1}^\infty u^*_i+v^*_i$, we deduce that $$\|j(y^*)-j(z^*)\|_{(X^{\ell_2})^*} \leq 2C\|y^*-z^*\|_{\ell_2(X)^*}.$$

\end{proof}

\begin{proof}[Proof of Proposition \ref{tech9}] It is easily seen that if $E$ and $F$ are Banach spaces, $B\subset E^*$ and $C\subset F^*$ are weak$^*$-compact and $f:B\to C$ is a Lipschitz surjection from $B$ to $C$, then $Sz(C)\le Sz(B)$ (see \cite[Lemma $2.5$(i)]{Causey2017a}). It follows from this fact and Claim \ref{lip} that $Sz(K) \leq Sz(M)$. On the other hand, since $M\subset B_{\ell_2(X)^*}$, we deduce from \cite{Brooker2011} that $Sz(M)\leq Sz(\ell_2(X))=Sz(X)$. Combining these yields that $Sz(K)\leq Sz(X)$. Denote by $L$ the weak$^*$-closed convex hull of $K$. It follows from Claim \ref{hb}  and the geometric Hahn-Banach theorem that $cB_{(X^{\ell_2})^*}\subset L \subset B_{(X^{\ell_2})^*}$. Finally we can apply Theorem 1.1 from \cite{Causey2017a} to deduce that since $Sz(K)\leq Sz(X)$, $Sz(L)\leq Sz(X)$. This finishes the proof of Proposition \ref{tech9}.
\end{proof}

The construction of our family of spaces $(G_\alpha)_{\alpha \in \Gamma\setminus \Lambda}$ will also rely on the use of the Schreier families. These families were introduced in \cite{AlspachArgyros1992}. Let us now recall the definition of the Schreier family $\mathcal{S}_\alpha$, for $\alpha$ a countable ordinal. Recall that $[\nn]^{<\omega}$ denotes the set of finite subsets of $\nn$, which we identify with the set of void or finite, strictly increasing sequences in $\nn$. We complete the notation introduced in section 2 by writing $E<F$ to mean $\max E<\min F$ and $n\leq E$ to mean $n\leq \min E$. For each countable ordinal $\alpha$,  $\mathcal{S}_\alpha$ will be a subset of $[\nn]^{<\omega}$. We let
$$\mathcal{S}_0=\{\emptyset\}\cup \{(n): n\in \nn\},$$
$$\mathcal{S}_{\alpha+1}= \{\emptyset\}\cup \Bigl\{\bigcup_{i=1}^n E_i:\ n\in \nn,\ \emptyset\neq E_i\in \mathcal{S}_\alpha,\  E_1<\ldots <E_n,\ n\leq E_1\Bigr\},$$
and if $\alpha<\omega_1$ is a limit ordinal, we fix an increasing sequence $(\alpha_n)_{n=1}^\infty$ tending to $\alpha$ and let $$\mathcal{S}_\alpha= \{E\in [\nn]^{<\omega}: \exists n\leq E\in \mathcal{S}_{\alpha_n}\}.$$
In what follows, $[\nn]^{<\omega}$ will be topologized by the identification $[\nn]^{<\omega}\ni E\leftrightarrow 1_E\in \{0, 1\}^\nn$, where $\{0, 1\}^\nn$ is equipped with the Cantor topology.\\
Given $(m_i)_{i=1}^k, (n_i)_{i=1}^k$ in $[\nn]^{<\omega}$, we say $(n_i)_{i=1}^k$ is a \emph{spread} of $(m_i)_{i=1}^k $ if $m_i\leq n_i$ for each $1\leq i\leq k$.

We say that a subset $\mathcal{F}$ of $[\nn]^{<\omega}$ is \begin{enumerate}[(i)]\item \emph{spreading} if it contains all spreads of its members, \item \emph{hereditary} if it contains all subsets of its members, \item \emph{regular} if it is spreading, hereditary, and compact. \end{enumerate}

Given $\mathcal{F}, \mathcal{G}\subset [\nn]^{<\nn}$, we let $$\mathcal{F}[\mathcal{G}]=\{\emptyset\} \cup \Bigl\{\bigcup_{i=1}^n E_i: \ n\in \nn,\ \emptyset\neq E_i\in \mathcal{G},\ E_1<\cdots <E_n,\ (\min E_i)_{i=1}^n\in \mathcal{F}\Bigr\}.$$
We refer to \cite{Causey2017b} for a detailed presentation of these notions and their fundamentals properties.

For a topological space $\cal F$, we denote $\cal F^{1}$ its Cantor-Bendixon derived set (the set of its accumulation points), for an ordinal $\alpha$, $\cal F^{\alpha}$, its Cantor-Bendixon derived set of order $\alpha$ and finally $CB(\cal F)$ its Cantor-Bendixon index.\\
We note that if $\cal F$ and $\cal G$ are regular subsets of $[\nn]^{<\omega}$, then $\mathcal{F}[\mathcal{G}]$ is regular and if the Cantor-Bendixson indices of $\mathcal{F}$ and $\mathcal{G}$ are $\alpha+1$ and $\beta+1$, respectively, then the Cantor-Bendixson index of $\mathcal{F}[\mathcal{G}]$ is $\beta\alpha+1$ (see Proposition $3.1$ in \cite{Causey2017b}).\\
For each $n\in \nn$, let
$$\mathcal{A}_n=\{E\in [\nn]^{<\omega}: |E|\leq n\}.$$
It is well-known that for each $\alpha<\omega_1$, $\mathcal{S}_\alpha$ is regular with Cantor-Bendixson index $\omega^\alpha+1$. Moreover, for each $n\in \nn$, $\mathcal{A}_n$ is regular with Cantor-Bendixson index $n+1$. These facts together with those cited from \cite{Causey2017b} yield the following.

\begin{lemma} Fix an ordinal $\alpha<\omega_1$ and $n\in \nn$. \begin{enumerate}[(i)]\item $\mathcal{A}_n[\mathcal{S}_\alpha]$ is regular with Cantor-Bendixson index $\omega^\alpha n+1$. \item For any $\beta<\omega_1$, $\mathcal{S}_\beta[\mathcal{S}_\alpha]$ is regular with Cantor-Bendixson index $\omega^{\alpha+\beta}+1$. \end{enumerate}

\label{complexity}

\end{lemma}

\begin{lemma} If $\mathcal{F}$ and $\mathcal{G}$ are regular families, $E<F\neq \emptyset$, and $E, E\cup F\in \mathcal{F}[\mathcal{G}]$, then either $E\in \mathcal{F}^1[\mathcal{G}]$ or $F\in \mathcal{G}$.

\label{fact}
\end{lemma}

\begin{proof} Write $E\cup F=\cup_{i=1}^n E_i$, $\emptyset\neq E_i\in \mathcal{G}$, $E_1<\ldots <E_n$, $(\min E_i)_{i=1}^n\in \mathcal{F}$.

If $E\cap E_n=\emptyset$, then there exists $1\leq m\leq n$ such that $E\cap E_i\neq \emptyset$ for each $i<m$ and $E\cap E_i=\emptyset$ for each $m\leq i\leq n$.\\
If $m=1$, $E=\emptyset\in \mathcal{F}^1$, since $\emptyset\prec (\min E_i)_{i=1}^n\in \mathcal{F}$. \\
If $m>1$, the representation
$$E=\bigcup_{i=1}^{m-1} (E\cap E_i)$$
witnesses that $E\in \mathcal{F}^1[\mathcal{G}]$, since $(\min E_i)_{i=1}^{m-1}\in \mathcal{F}^1$.

Now if $E\cap E_n\neq \emptyset$, then $F=E_n\setminus E\subset E_n$, and $F\in \mathcal{G}$.
\end{proof}

We are now ready to prove Theorem \ref{Causey2}, that is, to construct for each $\alpha\in \Gamma\setminus \Lambda$ a reflexive Banach space $G_\alpha$ with an unconditional basis and such that $Sz(G_\alpha)=\alpha$ and $Sz(G_\alpha^*)=\omega$.

So, let $\alpha\in \Gamma\setminus \Lambda$. We write $\alpha=\omega^\delta$, with $\delta \in (0,\omega_1)$. Then by standard facts about ordinals, either $\delta=\omega^\xi$ for some ordinal $\xi\in [0,\omega_1)$ or $\delta=\beta+\gamma$ for some $\beta, \gamma<\delta$. We shall separate our construction into these two main cases.

\subsection{First case: $\delta=\omega^\xi$}\

So let us first suppose that $\delta=\omega^\xi$ with $\xi\in [0,\omega_1)$. Then $\xi$ must either be $0$ or a successor ordinal, otherwise $\alpha\in \Lambda$.\\
If $\xi=0$, let $\mathcal{F}_n=\mathcal{S}_0$, for all $n \in \Ndb \cup\{0\}$.\\
If $\xi=\zeta+1$, let $\mathcal{F}_0=\mathcal{S}_0$ and $\mathcal{F}_{n+1}= \mathcal{S}_{\omega^\zeta}[\mathcal{F}_n]$ for $n\in \Ndb$.\\
In both cases, denote
$$M_n=\Bigl\{2^{-n} \sum_{i\in E} e^*_i:  E \in \mathcal{F}_n \Bigr\}\ \text{for}\ n\in \{0\}\cup \nn\ \
\text{and}\ \ M=\bigcup_{n=0}^\infty M_n.$$
where $(e_i^*)_{i=1}^\infty$ is the the sequence of coordinate functionals defined on $c_{00}$, the space of finitely supported sequences.\\
Then we define $\mathfrak{G}_\alpha$ to be the completion of $c_{00}$ with respect to the norm $$\|x\|_{\mathfrak{G}_\alpha} = \sup_{x^*\in M} |x^*(x)|.$$
Note that the canonical basis of $c_{00}$ is a 1-suppression unconditional basis of $\mathfrak{G}_\alpha$. To keep our notation consistent we shall denote $(x_i)_{i=1}^\infty$ this basis of $\mathfrak{G}_\alpha$. The reason is that we need next to set  $G_\alpha=\mathfrak{G}^{\ell_2}_\alpha$, where this construction is meant with respect to the basis $(x_i)_{i=1}^\infty$, which we shall later call the canonical basis of $\mathfrak{G}_\alpha$. On the other hand $(e_i)_{i=1}^\infty$ will still denote the canonical basis of $c_{00}$ considered as a basis of $G_\alpha$. Finally, we define the following subsets of $\mathfrak{G}_\alpha^*$:
$$K_n=\Bigl\{2^{-n} \sum_{i\in E} x^*_i:  E \in \mathcal{F}_n \Bigr\}\ \text{for}\ n\in \{0\}\cup \nn\ \
\text{and}\ \ K=\bigcup_{n=0}^\infty K_n.$$
Later, the sets $M_n$ and $M$ will be considered as subsets of $G_\alpha^*$.

\medskip It is easily checked that $\mathfrak{G}_\omega=c_0$ and $G_\omega=\ell_2$. So we clearly have that $G_\omega$ is reflexive with an unconditional basis and $Sz(G_\omega)=Sz(G_\omega^*)=\omega$. So we shall now assume that $\xi\neq 0$ and is therefore a countable successor ordinal.

\begin{proposition}\label{firstcase}  Assume $\alpha=\omega^{\omega^\xi}$, where $\xi$ is a countable successor ordinal.\\
Then, $Sz(\mathfrak{G}_\alpha)\leq \alpha$.

\label{first prop}

\end{proposition}

\begin{proof} By \cite[Theorem $1.1$]{Causey2017a}, it is sufficient to prove that $Sz(K)\leq \alpha$, since $B_{\mathfrak{G}_\alpha^*}$ is the weak$^*$-closed, absolutely convex hull of $K$.

First, it is easy to see that for any $\eps>0$ and any ordinal $\eta$, $$s^\eta_\eps(K)\subset \{0\}\cup\bigcup_{n=0}^\infty s^\eta_\eps(K_n),$$ whence $$Sz(K, \eps)\leq \Bigl(\sup_{n\in \nn\cup \{0\}} Sz(K_n, \eps)\Bigr)+1.$$ Thus it suffices to show that $\sup_{n\in \nn\cup \{0\}} Sz(K_n, \eps)<\alpha$ for each $\eps>0$.

For a given $\eps>0$, we will provide an upper estimate for $Sz(K_n, 2\eps)$ in one of two ways, depending on whether $n$ is large or small relative to $\eps$.     The Cantor-Bendixson index of $K_n$ is an easy upper bound for $Sz(K_n, 2\eps)$, which is a good upper bound for small $n$.    We note that the map $\phi_n:\mathcal{F}_n\to K_n$ given by $\phi_n(E)=\sum_{i\in E} x^*_i$ is a homeomorphism from $\mathcal{F}_n$ to $K_n$, where $K_n$ is endowed with its weak$^*$ topology. From this it follows that for any $n\in \nn\cup \{0\}$ and any $\eps>0$,
$$Sz(K_n, \eps) \leq CB(K_n) =CB(\mathcal{F}_n).$$

We now turn to bounding $Sz(K_n, 2\eps)$ for large $n$.     Recall that $\xi=\zeta+1$ with $\zeta \in [0,\omega_1)$. We now prove that if $2^{-m}<\eps$, then for any $n> m$ and any ordinal $\eta$:
$$s^\eta_{2\eps}(K_n) \subset \Bigl\{2^{-n}\sum_{i\in E} x^*_i: E\in \mathcal{F}_m^\eta[\mathcal{F}_{n-m}]\Bigr\}.$$
The proof is by induction on $\eta$, with the base case following from the fact that for any $a,b\in \nn$, $\mathcal{F}_a[\mathcal{F}_b]=\mathcal{F}_{a+b}$. The limit ordinal case follows by taking intersections. Finally, assume we have the result for some $\eta$ and $$2^{-n}\sum_{i\in E} x^*_i\in s^{\eta+1}_{2\eps}(K_n),$$
so that the inductive hypothesis guarantees that $E\in \mathcal{F}_m^\eta[\mathcal{F}_{n-m}]$.   Then there exists a sequence
$$\big(2^{-n}\sum_{i\in E_j} x^*_i\big)_{j=1}^\infty\subset s^\eta_{2\eps}(K_n, \eps) \subset  \Bigl\{2^{-n}\sum_{i\in E} x^*_i: E\in \mathcal{F}_m^\eta[\mathcal{F}_{n-m}]\Bigr\}$$ converging weak$^*$ to $2^{-n}\sum_{j\in E} x^*_i$ and such that
$$\liminf_{j \to \infty}\Big\|2^{-n}\sum_{i\in E} x_i^*-2^{-n}\sum_{i\in E_j}x_i^*\Big\|_{\mathfrak{G}_\alpha^*}\ge \eps.$$
Of course, this means that $E_j\to E$ in $\mathcal{F}_n$ so that, after passing to another subsequence, we may assume $E_j=E\cup F_j$ for some $F_j\neq \emptyset$ with $E<F_j$.  Now since $E, E_j\in \mathcal{F}^\eta_m[\mathcal{F}_{n-m}]$ for each $j$, by Lemma \ref{fact}, either  $F_j\in \mathcal{F}_{n-m}$ or $E\in \mathcal{F}_m^{\eta+1}[\mathcal{F}_{n-m}]$. However, if $F_j\in \mathcal{F}_{n-m}$, then $2^{m-n}\sum_{i\in F_j} x_i^*\in B_{\mathfrak{G}_\alpha^*}$ and
$$\forall j \in \nn\ \ \Big\|2^{-n}\sum_{i\in E} x_i^*- 2^{-n}\sum_{i\in E_j} x^*_i\Big\|_{\mathfrak{G}_\alpha^*}= 2^{-m}\Big\|2^{m-n}\sum_{i\in F_j} x_i^*\Big\|_{\mathfrak{G}_\alpha^*} \leq 2^{-m}<\eps,$$
a contradiction. This concludes the successor case.\\
We now deduce from the inclusion we just proved, that
$$s^{\omega^{\omega^\zeta m}+1}_{2\eps}(K_n) \subset \Bigl\{2^{-n}\sum_{i\in E} x^*_i: E\in \mathcal{F}_m^{\omega^{\omega^\zeta m}+1}[\mathcal{F}_{n-m}]\Bigr\}=\emptyset.$$
So, we can now estimate
\begin{displaymath}
   Sz(K_n, 2\eps) \leq \left\{
     \begin{array}{lr}
       \omega^{\omega^\zeta n}+1 & : n\leq \log_2(1/\eps)\\
        \omega^{\omega^\zeta \lceil \log_2(1/\eps)\rceil}+1 & :n> \log_2(1/\eps),
     \end{array}
   \right.
\end{displaymath}
and this estimate finishes the proof of our proposition.

\end{proof}

\begin{proof}[Proof of Theorem \ref{Causey2} in the first case] Let $\alpha=\omega^{\omega^\xi}$, where $\xi$ is a countable successor ordinal and $\mathfrak{G}_\alpha$, $G_\alpha$ constructed as above.\\
Since the canonical basis $(x_i)_{i=1}^\infty$ of $\mathfrak{G}_\alpha$ is $1$-suppression unconditional, it is clear that $(e_i)_{i=1}^\infty$ is a $1$-suppression unconditional basis for $G_\alpha$. Proposition \ref{firstcase} insures that $Sz(\mathfrak{G}_\alpha)\leq\alpha$ and therefore that $\mathfrak{G}_\alpha$ does not contain $\ell_1$. It then follows from a classical result of R.C. James \cite{James1950} that $(x_i)_{i=1}^\infty$ is a shrinking basis of $\mathfrak{G}_\alpha$. Thus we can apply Proposition \ref{OSZreflexive} to get that $G_\alpha$ is reflexive and $Sz(G_\alpha^*)=\omega$.\\
We also deduce from Proposition \ref{tech9} that $Sz(G_\alpha)\leq Sz(\mathfrak{G}_\alpha)=\alpha$.

We now have to prove that $Sz(G_\alpha)\ge \alpha$. So let us write again $\alpha=\omega^{\omega^{\zeta+1}}$, with $\zeta \in [0,\omega_1)$.   Suppose  $n\in \nn$ and $E<F$ are such that $F\in \mathcal{F}_n$. Fix $k\in F\setminus E$.  Note that  $$2^{-n}\sum_{i\in F}e_i^* \in M_n$$
and
$$\Big\|2^{-n}\sum_{i\in E}e_i^* - 2^{-n}\sum_{i\in F} e_i^*\Big\|_{G_\alpha}\geq \Bigl|\Bigl(2^{-n}\sum_{i\in E}e_i^* - 2^{-n}\sum_{i\in F} e_i^*\Bigr)(e_k)\Bigr|=2^{-n},$$ since $\|e_k\|_{G_\alpha}=1$. From this and an easy induction argument, we see that for any $n\in \nn$, any $0\leq \mu<CB(\mathcal{F}_n)$ and any $E\in \mathcal{F}_n^\mu$, $2^{-n}\sum_{i\in E}e_i^*\in s_{2^{-n-1}}^\mu(B_{G_\alpha^*})$. Since $CB(\mathcal{F}_n)=(\omega^{\omega^\zeta})^n= \omega^{\omega^\zeta n}$, we deduce that $$Sz(G_\alpha)\geq \sup_{n\in \nn} \omega^{\omega^\zeta n}=\omega^{\omega^{\zeta+1}}=\alpha.$$
This finishes the proof and our construction for $\alpha=\omega^{\omega^\xi}$, with $\xi$ being a countable successor ordinal.

\end{proof}

\subsection{Second case: $\delta=\beta+\gamma$ for some $\beta, \gamma<\delta$}\

We will now modify slightly our construction in order to treat the case in which $\alpha=\omega^{\beta+\gamma}$, with $\omega^\beta<\alpha$ and $\omega^\gamma<\alpha$.  We have to consider two subcases.\\
First suppose $\gamma$ is a limit ordinal. We fix $\gamma_0=0$ and an increasing sequence $(\gamma_n)_{n=1}^\infty$ such that $\sup_{n\in \nn} \gamma_n=\gamma$. Then we set $$\mathcal{F}_0= \mathcal{S}_\beta\ \  \text{and}\ \ \mathcal{F}_n=\mathcal{S}_{\gamma_n}[\mathcal{S}_\beta],\ \text{for}\ n\in \nn.$$
If $\gamma=\zeta+1$ is a successor ordinal, we set
$$\mathcal{F}_0=\mathcal{S}_{\beta+\zeta} \  \text{and}\ \ \mathcal{F}_n=\mathcal{A}_n[\mathcal{S}_{\beta+\zeta}],\ \text{for}\ n\in \nn.$$
In either case, let
$$M_n=\Bigl\{2^{-n} \sum_{i\in E} e^*_i:  E \in \mathcal{F}_n \Bigr\}\ \text{for}\ n\in \{0\}\cup \nn\ \
\text{and}\ \ M=\bigcup_{n=0}^\infty M_n.$$
As in our first situation, we define $\mathfrak{G}_\alpha$ to be the completion of $c_{00}$ with respect to the norm $\|x\|_{\mathfrak{G}_\alpha}=\sup_{x^*\in M} |x^*(x)|$ and let $G_\alpha=\mathfrak{G}_\alpha^{\ell_2}$, where this construction is meant with respect to the canonical basis $(x_i)_{i=1}^\infty$ of $\mathfrak{G}_\alpha$. As previously, we define
$$K_n=\Bigl\{2^{-n} \sum_{i\in E} x^*_i:  E \in \mathcal{F}_n \Bigr\}\ \text{for}\ n\in \{0\}\cup \nn\ \
\text{and}\ \ K=\bigcup_{n=0}^\infty K_n.$$

\begin{proposition} Assume that $\alpha$ is a countable ordinal that can be written  $\alpha=\omega^{\beta+\gamma}$, with $\omega^\beta<\alpha$ and $\omega^\gamma<\alpha$. Then  $Sz(\mathfrak{G}_\alpha) \leq \alpha$.
\end{proposition}

\begin{proof} Again, it is sufficient to show that $Sz(K)\leq \alpha$. Arguing as in Proposition \ref{first prop}, we first note that for any $\eps>0$ and $n\in \nn$,

 \begin{displaymath}
   Sz(K_n, \eps) \leq CB(\mathcal{F}_n)  = \left\{
     \begin{array}{lr}
       \omega^{\beta+\gamma_n}+1 & : \gamma \text{\ a limit}\\
       \omega^{\beta+\mu}n +1 & :  \gamma=\zeta+1.
     \end{array}
   \right.
\end{displaymath}
Now for $n\in \nn$ and $\eps>0$ such that  $2^{-n}<\eps$, we claim that for any ordinal $\eta$,  \begin{displaymath}
   s^\eta_{2\eps}(K_n) \subset  \left\{
     \begin{array}{lr}
       \Bigl\{2^{-n}\sum_{i\in E} x^*_i: E\in \mathcal{S}^\eta_{\gamma_n}[\mathcal{S}_\beta]\Bigr\} & : \gamma \text{\ a limit}\\
       \Bigl\{2^{-n}\sum_{i\in E} x^*_i: E\in \mathcal{A}_n^\eta[\mathcal{S}_{\beta+\zeta}]\Bigr\} & :  \gamma=\zeta+1.
     \end{array}
   \right.
\end{displaymath}
The proof is even easier than the analogous claim in the proof of the first case, so we omit it. Note that in particular, when $\gamma$ is a limit ordinal and $2^{-n}<\eps$, $\mathcal{S}_{\gamma_n}^{\omega^\gamma}=\emptyset$, whence the previous claim yields the estimate $Sz(K_n, 2\eps)\leq \omega^\gamma<\omega^{\beta+\gamma}$ when $2^{-n}<\eps$. Similarly, since $\mathcal{A}_n^\omega=\emptyset$, $Sz(K_n, 2\eps)\leq \omega<\omega^{\beta+\zeta+1}$ when $2^{-n}<\eps$.\\
Therefore for $n\leq \log_2(1/\eps)$, \begin{displaymath}
   Sz(K_n, 2\eps) \leq CB(\mathcal{F}_n)  = \left\{
     \begin{array}{lr}
       \omega^{\beta+\gamma_n}+1 & : \gamma \text{\ a limit}\\
       \omega^{\beta+\mu}n +1 & :  \gamma=\zeta+1,
     \end{array}
   \right.
\end{displaymath}
and for $n>\log_2(1/\eps)$, \begin{displaymath}
   Sz(K_n, 2\eps) \leq  \left\{
     \begin{array}{lr}
       \omega^{\gamma} & : \gamma \text{\ a limit}\\
       \omega  & :  \gamma=\zeta+1,
     \end{array}
   \right.
\end{displaymath}
Thus in either case, for every $\eps>0$, $\sup_{n\in \nn\cup \{0\}} Sz(K_n, \eps)<\alpha$, yielding the result.

\end{proof}

\begin{proof}[Proof of Theorem \ref{Causey2} in the second case]
The end of the proof is the same as for the first case, only noting that $CB(\mathcal{F}_n)=\omega^{\beta+\gamma_n}+1$ when $\gamma$ is a limit ordinal and $CB(\mathcal{F}_n)=\omega^{\beta+\zeta}n+1$ if $\gamma=\zeta+1$.

\end{proof}

\medskip\noindent{\bf Acknowledgements.} This work was initiated while the second named author was visiting  Miami University in Oxford, Ohio. He wishes to thank the Mathematics Department of Miami University for this invitation and its hospitality.\\
The authors are grateful to the anonymous referee, whose valuable comments and suggestions for simplification helped to improve the presentation of this paper.

\end{document}